\documentclass[11pt]{amsart}
\usepackage{color}

\usepackage{amsmath,amssymb,amsthm,amscd}
\usepackage{graphicx,subfigure,hyperref}
\usepackage{pgfplots}

\makeatletter
\newcommand{\newreptheorem}[2]{\newtheorem*{rep@#1}{\rep@title}\newenvironment{rep#1}[1]{\def\rep@title{#2 \ref{##1}}\begin{rep@#1}}{\end{rep@#1}}}
\makeatother

\newtheorem{thm}{Theorem}[section]
\newtheorem*{thm*}{Theorem}
\newreptheorem{thm}{Theorem}
\newtheorem{lem}[thm]{Lemma}
\newtheorem*{lem*}{Lemma}

\newtheorem*{prop*}{Proposition}

\theoremstyle{definition}
\newtheorem{defn}[thm]{Definition}

\title{Linking number and writhe in random linear embeddings of graphs}

\author{Erica Flapan}
\address{Department of Mathematics,
610 N. College Ave.,
Pomona College,
Claremont, CA USA}
\email{elf04747@pomona.edu}

\author{Kenji Kozai}
\address{Department of Mathematics,
University of California, Berkeley,
970 Evans Hall \#3840,
Berkeley, CA 94720-3840}
\email{kozai@math.berkeley.edu}



\begin{document}

\begin{abstract}
In order to model entanglements of polymers in a confined region, we consider the linking numbers and writhes of cycles in random linear embeddings of complete graphs in a cube.  Our main results are that for a random linear embedding of $K_n$ in a cube, the mean sum of squared linking numbers and the mean sum of squared writhes are of the order of $\theta(n(n!))$.  We obtain a similar result for the mean sum of squared linking numbers in linear embeddings of graphs on $n$ vertices, such that for any pair of vertices, the probability that they are connected by an edge is $p$.  We also obtain experimental results about the distribution of linking numbers for random linear embeddings of these graphs.  Finally, we estimate the probability of specific linking configurations occurring in random linear embeddings of the graphs $K_6$ and $K_{3,3,1}$.

\end{abstract}

\maketitle

\section{Introduction}

Long polymers become tangled up as a result of being tightly packed in a confined region.  For example, 46 human chromosomes are packed together inside the nucleus of a cell whose diameter can be as little as $10^{-5}$ times the length of a single chromosome.  The entanglement that results affects the processes of replication and transcription of the DNA.  For synthetic polymers, tangling is correlated with viscoelastic properties, and hence is important in the design and synthesis of new elastic materials.  As these examples illustrate, understanding the tangling of polymers is useful for explaining and controlling molecular behaviour.  However, since detailed visualizations of molecular entanglements are not yet technologically possible, their study has been approached through mathematical modeling rather than experimental observation.

Many authors have considered uniform random distributions of open and closed polygonal chains in a cube as a model for long molecular chains in a confined region (see for example \cite{arsuaga07}, \cite{Arsuaga09}, \cite{Diao93}, \cite{Diao}, \cite{panagiotou10},  \cite{Portillo}, \cite{Numerical}).  Of particular note, Arsuaga et al \cite{arsuaga07} obtained a formula for the mean squared linking number of two uniform random $n$-gons in a cube, and showed that the probability of linking between a given simple closed curve in the cube and a uniform random $n$-gon grows at a rate of at least $1-\mathcal{O}\left(\frac{1}{\sqrt{n}}\right)$.  More recently, Panagiotou et al \cite{panagiotou10} has shown that the mean squared linking number, the mean squared writhe, and the mean squared self-linking number of oriented uniform random open or closed chains with $n$ vertices in a cube all grow at a rate of $\mathcal{O}(n^2)$. 

 Many of the results that have been obtained about the linking of two random uniform polygons in a cube are restricted to pairs of polygons of the same length.  However, there is no biological reason for all polymers in a given region to have the same length.  While no theoretical results have been proven thus far about linking between uniform random $n$- and $m$-gons in a confined region, Arsuaga et al \cite{arsuaga07} has observed from numerical simulations that the linking probability of a random linear $n$-gon and $m$-gon in a cube seems to be bounded below by $1-\mathcal{O}\left(\frac{1}{\sqrt{nm}}\right)$.  
 
In order to obtain theoretical results about linking between random chains of different lengths as well as to measure entanglement in a more general way, we take a new approach.  In particular, the above models use {\it ordered} sequences of $n$ points chosen from a uniform random distribution of points in a cube to define one or two $n$-gons with linear edges.  By contrast, we begin with an {\it unordered} set of $n$ points chosen from a uniform random distribution of points in a cube.  We then consider every possible pair of disjoint polygons obtained by adding line segments between some number of points in the set.   By taking the sum of the squared linking numbers of all such pairs of polygons, we obtain a single number which represents the linking of all pairs of polygons with vertices in this set regardless of whether the polygons are of the same or distinct lengths.  In addition, in order to measure the entanglement of individual polymers in a confined region, we consider the sum of the mean squared writhes over all polygons with vertices in our set.
 
In addition to modeling entanglement of polymers in confined regions, our results can be seen in the context of  linear embeddings of graphs in $\mathbb{R}^3$ (that is, embeddings whose edges are realized by straight line segments).  In particular, the set of polygons we are considering are the cycles in a linear embedding of the complete graph $K_n$ in $\mathbb{R}^3$. Probably the most significant result in the study of such embeddings was the proof by Negami \cite{negami91} that for every knot or link $J$, there is an integer $R(J)$ such that every linear embedding of the complete graph $K_{R(J)}$ in $\mathbb{R}^3$ contains $J$.  In addition, several authors have obtained results characterizing what links can occur in linear embeddings of specific graphs.  In particular, Hughes \cite{Hughes} and Huh and Jeon \cite{huh07} gave combinatorial proofs that every linear embedding of $K_6$ contains either one or three Hopf links and no other links.  More recently, Nikkuni \cite{Nikkuni}  obtained the same result with a topological proof.  Naimi and Pavelescu \cite{naimi14} proved that every linear embedding of $K_9$ contains a non-split link of three components, and showed in \cite{naimi15} that every linear embedding of $K_{3,3,1}$ contains either $1,2,3,4$, or $5$ non-trivial links.

Our main results concern the rate of growth of two measures of entanglement.  Since rates of growth can be measured in several ways, for clarity we make the following definitions. 

\begin{defn} Let $f(n)$ be a function of the naturals.

\begin{itemize} 
	\item $f(n)$ is said to be of the order of $\mathcal{O}(g(n))$ if there exists a constant $C>0$
	such that for sufficiently large $n$, $$f(n) \leq C g(n).$$

	\item $f(n)$ is said to be of the order of $\theta(g(n))$, if there exist constants $c$, $C > 0$ such that
	for sufficiently large $n$, $$c g(n) \leq f(n) \leq C g(n).$$ 
	\end{itemize}
	
\end{defn}
\medskip

Section 2 is devoted to the proofs of the following two theorems about entanglement of {\it random linear embeddings} of complete graphs inside a cube $C^3=[0,1]^3$. That is, embeddings of complete graphs whose vertices are given by a random uniform distribution of $n$ points in the cube and whose edges are realized by straight line segments.

\begin{repthm}{thm:ev_ss_link} Let $n\geq 6$, and let $K_n$ be a random linear embedding of the complete graph on $n$ vertices in the cube $C^3$.  Then the mean sum of squared linking numbers for $K_n$ is of the order of $\theta(n(n!))$.
\end{repthm}
 
\begin{repthm}{thm:Wr} Let $n\geq 3$, and let $K_n$ be a random linear embedding of the complete graph on $n$ vertices in the cube $C^3$.   Then the mean sum of squared writhe for $K_n$ is of the order of 
	$\theta(n(n!))$.
\end{repthm} 

The ideas of the proofs of these results are as follows. The first theorem is proved using Lemma \ref{lem:ev_mn_link}, which shows that the expected linking
number of random cycles of length $k$ and length $l$ is on the order of $\theta(kl)$.
The complete graph $K_n$ contains on the order of $\theta((n-1)!)$ links,
both of whose cycles are of length $\frac{n}{2}$.  Heuristically, the linking number from these
cycles dominate the sum of squared linking numbers, resulting in a mean
sum of squared linking numbers on the order of $\theta(n(n!))$. To prove the second theorem, we observe that there
are many more cycles of length $n$ than of any other length, and the number of such cycles is on the
order of $\theta((n-1)!)$.  We then use the result of Panagiotou et al \cite{panagiotou10} that the mean squared writhe of
an $n$-cycle is on the order of $\theta(n^2)$ to obtain the desired result.


In Section 3, we consider a set of $n$ points chosen from a uniform random distribution of points in a cube, and then assign a probability that a given pair of vertices is joined with an edge.  In this way, we obtain a subgraph of $K_n$ with a given probability.  In particular, we use the following definition, originally due to Gilbert \cite{gilbert59}.

\begin{defn}
	Let $n \in \mathbb{N}$ and $p \in (0,1)$. Then a {\it $(n,p)$-graph} is
	a graph on $n$ vertices, such that for any pair of vertices, the probability that
	they are connected by an edge is $p$.
\end{defn}

We obtain several results on knotting and linking in random linear embeddings of $(n,p)$-graphs, including the following.

\begin{repthm}{thm:randomknot}
	For any knot $J$, the probability that a random linear embedding
	of an $(n,p)$-graph contains a cycle isotopic to $J$ goes to 1 as $n \rightarrow
	\infty$.
\end{repthm}

\begin{repthm}{thm:nplink}
	For any $n\geq 6$ and $p\in (0,1)$, the mean sum of squared linking numbers of a random linear embedding of an $(n,p)$-graph is of the order of $\theta(p^nn(n!))$.
\end{repthm}

In Section 4, we apply the inequalities that we obtain in Section 2 to random linear embeddings of the graphs $K_6$ and $K_{3,3,1}$ in order to estimate the probability of specific linking configurations occurring.

Finally, in Section 5, we sample random linear graph embeddings and make some observations about the distribution of linking numbers.

\bigskip

 \section{Random linear embeddings of $K_n$}
 \medskip
 
\begin{defn}
	A \textit{random linear embedding} of a graph $G$, is an embedding of $G$ in the unit cube  $C^3 = [0,1]^3$ such that the vertices of $G$ are embedded with a
	uniform distribution, and every edge $(v_i,v_j)$ of $G$ is realized by a straight line
	segment between $v_i$ and $v_j$.
\end{defn}

Arsuaga et al \cite{arsuaga07} prove the following lemma.

\begin{lem}[\cite{arsuaga07}]\label{lem:q}
	Let $l_1,l_2,l_1',l_2'$ denote edges in a random linear embedding of a graph with
	an orientation assigned to each edge.  Let $\epsilon_i$ denote the signed crossing of $l_i$ and $l_i'$, and let $E[\epsilon_1\epsilon_2]$ denote the expected value of $\epsilon_1\epsilon_2$.
	\begin{enumerate}
		\item If the endpoints of $l_1,l_2,l_1',l_2'$ are distinct, then
			$E[\epsilon_1\epsilon_2] = 0$.
		\item If $l_1=l_2$, the endpoints of $l_1'$ and $l_2'$ are distinct, and
			both $l_1',l_2'$ are disjoint from $l_1=l_2$,
			then $E[\epsilon_1\epsilon_2] = 0$.
		\item Define variables as follows:
		\begin{itemize}
		\item 		Let $2s$ denote the probability that $l_1$ and $l_1'$ cross when
			$l_1$ and $l_1'$ are disjoint.
			
\item  Let $u=E[\epsilon_1\epsilon_2]$ when
			$l_1 = l_2$, $l_1'$ and $l_2'$ share exactly one endpoint,
			and $l_1'\cup l_2'$ is disjoint from $l_1=l_2$.
			
			\item Let $v=E[\epsilon_1\epsilon_2]$ when $l_1$ and $l_2$ share exactly one endpoint, $l_1'$ and $l_2'$ share exactly
			one endpoint, and $l_1 \cup l_2$ and $l_1' \cup l_2'$ are
			disjoint. 
		
		\end{itemize}
		
			 Then, $q=s+2(u+v)>0$.
	\end{enumerate}
\end{lem}

Arsuaga et al \cite{arsuaga07} use the above lemma to prove that the mean squared linking number of two uniform random polygons of length $n$ is $\frac{1}{2}n^2q$ where $q$ is defined in Case (3) of the lemma.  We now apply the above lemma in a similar way to obtain a formula for the mean squared linking number of two uniform random polygons where the number of vertices in the two polygons may differ.

\begin{lem}\label{lem:ev_mn_link}
	Let $n$, $m\ge 3$, and let the graph $G$ be the disjoint union of an $n$-cycle $L$ and an $m$-cycle $L'$.
	Then the mean squared linking number of a random linear embedding of 
	$G$ in the cube $C^3$ is $\frac{1}{2}nmq$, where $q$ is defined in Case (3) of Lemma \ref{lem:q}.
\end{lem}

\begin{proof}
	Let the edges of $L$ be $l_1,l_2,\dots,l_n$ and the edges of $L'$ be $l_1', l_2',\dots,l_m'$, both cyclically ordered
	and oriented. 
	Let $\epsilon_{ij}$ denote the signed crossing of $l_i$ and $l_j'$. Then,
	the linking number of $L$ and $L'$ is given by
	\begin{equation*}
		lk(L,L') = \frac{1}{2} \sum_{i=1}^n \sum_{j=1}^m \epsilon_{ij}.
	\end{equation*}
	\medskip
	
	Hence the expected value of the mean squared linking number is given by:
	\medskip
	
	\begin{align*}
		E \left[ \left(\frac{1}{2} \sum_{i=1}^n \sum_{j=1}^m \epsilon_{ij}\right)^2\right]
			=& \frac{1}{4} E\left[\left(\sum_{i=1}^n \sum_{j=1}^m \epsilon_{ij}\right)^2\right]\\
			=& \frac{1}{4} \sum_{i=1}^n\sum_{j=1}^m E[\epsilon_{ij}^2] +
				\frac{1}{2}\sum_{i=1}^n\sum_{j=1}^m (E[\epsilon_{ij}\epsilon_{i(j-1)}]
				+E[\epsilon_{ij}\epsilon_{i(j+1)}])\\
				& + \frac{1}{2}\sum_{i=1}^n\sum_{j=1}^m (E[\epsilon_{ij}\epsilon_{(i+1)(j+1)}]
				+E[\epsilon_{ij} \epsilon_{(i-1)(j+1)}]).
	\end{align*}
	Note that those cross terms which we know are $0$ by Cases (1) and (2) of Lemma \ref{lem:q} have been omitted from the above expansion. 
	
	Let $2s$ denote the probability that a pair of edges $l_i$ and $l_j'$ cross.  Then the first term in the above expansion is $$ \frac{1}{4} \sum_{i=1}^n\sum_{j=1}^m E[\epsilon_{ij}^2] =\frac{1}{2}nms.$$ 
	
	 Let $u$ denote the expected value of the product of signed crossings of an edge $l_i$ of $L$ with consecutive edges $l_j'$ and $l_{j\pm1}'$ of $L'$.  Then the second term in the above expansion is given by $$\frac{1}{2}\sum_{i=1}^n\sum_{j=1}^m (E[\epsilon_{ij}\epsilon_{i(j-1)}]
				+E[\epsilon_{ij}\epsilon_{i(j+1)}])=nmu.$$

	 Let $v$ denote the expected value of the product of signed crossings of consecutive edges $l_i$ and $l_{i\pm 1}$ of $L$ with consecutive edges $l_j'$ and $l_{j+1}'$ respectively of $L'$. Then the third term in the above expansion is given by $$\frac{1}{2}\sum_{i=1}^n\sum_{j=1}^m (E[\epsilon_{ij}\epsilon_{(i+1)(j+1)}]
				+E[\epsilon_{ij} \epsilon_{(i-1)(j+1)}])=nmv$$

	Finally, let $q=s+2(u+v)$.  Then we have
		
	\begin{equation*}
		E \left[ \left(\frac{1}{2} \sum_{i=1}^n \sum_{j=1}^m \epsilon_{ij}\right)^2\right]
			= \frac{1}{2}nms + nmu + nmv = \frac{1}{2} nm(s+2(u+v))=\frac{1}{2} nmq.
	\end{equation*}
\end{proof}
\medskip

Using the above lemma, we now prove the following theorem.

\begin{thm}\label{thm:ev_ss_link} Let $n\geq 6$, and let $K_n$ be a random linear embedding of the complete graph on $n$ vertices in the cube $C^3$.  Then the mean sum of squared linking numbers for $K_n$ is of the order of $\theta(n(n!))$.
\end{thm}

\begin{proof}  Let $k$, $l\geq 3$ such that $k+l\leq n$.  If  $k \neq l$, then the number of disjoint pairs of cycles in $K_n$ such that one cycle has $k$ vertices and the other cycle has $l$ vertices is given by

	\begin{equation*}
		{n \choose k}{n-k \choose l}\frac{(k-1)!}{2}\frac{(l-1)!}{2}.
	\end{equation*}
	\medskip
	
	If $k=l$, then this number is given by
	
	\begin{equation*}
		\frac{1}{2} {n \choose k}{n-k \choose l}\frac{(k-1)!}{2}\frac{(l-1)!}{2},
	\end{equation*}
	\medskip
	
	By Lemma \ref{lem:ev_mn_link}, we know that the mean squared linking number of
	a $k$-cycle and an $l$-cycle in $K_n$ is $\frac{1}{2}klq$, where $q$ is defined in Case (3) of Lemma \ref{lem:q}. 
	Thus, we obtain the mean sum of squared
	linking numbers over all disjoint pairs of cycles in $K_n$ as 
		
	\begin{align*}
		\frac{q}{4} \sum_{k=3}^{n-3} \sum_{l=3}^{n-k}
			& kl{n \choose k}{n-k \choose l}\frac{(k-1)!}{2}\frac{(l-1)!}{2}\\
		&=\frac{q}{4} \sum_{k=3}^{n-3} \sum_{l=3}^{n-k}
			{n \choose k}{n-k \choose l}\frac{k!}{2}\frac{l!}{2}\\
		&=\frac{q}{16} \sum_{k=3}^{n-3} \sum_{l=3}^{n-k} \frac{n!}{(n-k-l)!}.
	\end{align*}
	
	\medskip
	
	Observe that the double sum $$\sum_{k=3}^{n-3} \sum_{l=3}^{n-k} \frac{n!}{(n-k-l)!}$$
	counts the number of ways to obtain disjoint subsets of $k\geq 3$ and $l\geq 3$ ordered points
	from the set of $n$ points. This same quantity can alternatively be counted by choosing an ordered list of
	$i=k+l$ points out of $n$, then picking a number $3 \leq j \leq i-3$, so that
	the first $j$ points are in the subset of $k$ points, and the rest are in the subset of
	$l$ points. Hence, we have the equality:
	
	\begin{equation*}
		\sum_{k=3}^{n-3} \sum_{l=3}^{n-k} \frac{n!}{(n-k-l)!}= \sum_{i=6}^n \frac{n!}{(n-i)!}(i-5).
	\end{equation*}
	\medskip
	
	If we only consider the $i=n$ term in the sum, we obtain the following lower bound for the mean sum of squared linking numbers.
	
	\begin{equation*}
		\frac{q}{16} \sum_{i=6}^n \frac{n!}{(n-i)!}(i-5) \geq \frac{q}{16}\frac{n!}{0!} (n-5)= \frac{q}{16}(n-5)n!.
	\end{equation*}
	\medskip
	
	For sufficiently large $n$, we have $n-5>\frac{n}{2}$.  Thus we have the lower bound
	
	\begin{equation*}
		\frac{q}{16} \sum_{i=6}^n \frac{n!}{(n-i)!}(i-5) \geq  \frac{q}{32}(n)n!.
	\end{equation*}
	\medskip
	
	For an upper bound, we find that,
	
	\begin{align*}
		 \frac{q}{16}\sum_{i=6}^n \frac{n!}{(n-i)!}(i-5)  &\leq\frac{q}{16} n! \sum_{i=6}^n \frac{n}{(n-i)!}\\
			&= \frac{q}{16}n(n!)\sum_{i=6}^n \frac{1}{(n-i)!}\\
			&\leq \frac{q}{16}n(n!) \sum_{m=1}^\infty \frac{1}{m!}\\
			&=\frac{q}{16}n(n!)e.
	\end{align*}
	
		\medskip
	
	Putting these inequalities together, we see that the mean sum of squared linking numbers is of the order of $\theta(n(n!))$.
\end{proof}
\medskip

Another way to model entanglement is to consider the tangling of individual cycles rather than the linking between cycles.  In particular, given a fixed oriented $k$-cycle $J_k$ in $\mathbb{R}^3$, we define the {\it directional writhe} $\text{Wr}_\xi(J_k)$ projected in a direction perpendicular to a unit vector $\xi\in S^2$ as the algebraic sum of the signed crossings of $J_k$.  In order to avoid issues of sign, it is preferable to work instead with the {\it directional squared writhe}, which is defined as $\text{Wr}^2_\xi(J_k)=(\text{Wr}_\xi(J_k))^2$.  Now if we average the directional squared writhe $\text{Wr}^2_\xi(J_k)$ over all possible direction vectors $\xi\in S^2$, we obtain the {\it mean squared writhe} denoted by $\text{Wr}^2(J_k)$.  More formally, we define 
\begin{equation*}
		\text{Wr}^2(J_k)= \frac{1}{4\pi}\int_{S^2} \text{Wr}_\xi^2(J_k)d\xi
		\end{equation*}

\medskip

 Panagiotou et al \cite {panagiotou10} prove that the mean squared writhe of a random linear embedding of a $k$-cycle is of the order of $\mathcal{O}(k^2)$.  Rather than focusing on a single cycle, we are interested in obtaining a single value representing the complexity of the entanglement of all cycles $C$ in a random linear embedding of $K_n$.  Thus we define the {\it mean sum of squared writhe} of $K_n$ as the expected value 
 
 $$E[\sum_{C\subseteq K_n} \text{Wr}^2(C)]$$
 \medskip
 
 \noindent over all random linear embeddings of $K_n$.   
 
 \medskip
 
 We will make use of the following lemma from \cite{panagiotou10} which is similar to Lemma \ref{lem:q}.

\begin{lem}[\cite{panagiotou10}] \label{lem:q'}
	Let $l_1,l_2,l_1',l_2'$ denote edges in a random linear embedding of a graph with
	an orientation assigned to each edge, let $\epsilon_i$ denote the signed crossing of $l_i$ and $l_i'$, and let $E[\epsilon_1\epsilon_2]$ denote the expected value of $\epsilon_1\epsilon_2$.
 Also, let $s$, $u$, and $v$ be defined in Case (3) of Lemma \ref{lem:q}, and let
	$w=E[\epsilon_1\epsilon_2]$ when $l_1,l_2,l_1',l_2'$ are consecutive
	edges. Then $q'=3s+2(2u+v+w) > 0$.
\end{lem}

\medskip
 
\begin{thm}\label{thm:Wr} Let $n\geq 3$, and let $K_n$ be a random linear embedding of the complete graph on $n$ vertices in the cube $C^3$.   Then the mean sum of squared writhe for $K_n$ is of the order of 
	$\theta(n(n!))$.
\end{thm}

\begin{proof}  For some $k\leq n$,  let $J_k$ be a $k$-cycle in $K_n$.   It follows from Panagiotou et al  \cite{panagiotou10} that the mean squared writhe satisfies	
	
	$$\text{Wr}^2(J_k)=qk^2-(6q-q')k$$

\medskip

\noindent	where $q$ is defined in Case (3) of Lemma \ref{lem:q}.   Also, by Lemma~\ref{lem:q'} we know that $q'>0$.  Thus we have

$$\text{Wr}^2(J_k)>qk^2-6qk.$$  
\medskip

	\noindent  Hence, we have the following lower bound for the mean sum of squared writhe.
	
	\begin{equation*}
		E[\sum_{J\subseteq K_n} \text{Wr}^2(J)]\geq \sum_{k=3}^n (qk^2-6qk)
			\frac{n!}{(n-k)!(2k)}=\frac{qn!}{2}\sum_{k=3}^n\frac{k-6}{(n-k)!}
	\end{equation*} 
	\medskip
			
	\noindent  Taking only the term when $k=n$, we see that	
	\begin{equation*}
		\frac{qn!}{2}\sum_{k=3}^n\frac{k-6}{(n-k)!} \geq \left(\frac{qn!}{2} \right)(n-6).
	\end{equation*}

	For sufficiently large $n$, we have $n-6>\frac{n}{2}$.  Thus we obtain the lower bound
	
		\begin{equation*}
	E[\sum_{J\subseteq K_n} \text{Wr}^2(J)] \geq \left(\frac{q}{4} \right)(n)n!.
	\end{equation*}

	In order to get an upper bound, first observe that for any $k$-cycle $J_k$, by Panagiotou et al \cite{panagiotou10} we have
			
	$$\text{Wr}^2(J_k)=qk^2-(6q-q')k$$
	
	\medskip 
	
	\noindent Thus, taken over all cycles $J$ in $K_n$, we have the expected value
	
	$$E[\sum_{J\subseteq K_n} \text{Wr}^2(J)]= \sum_{k=3}^n (qk^2-(6q-q')k) \frac{n!}{(n-k)!(2k)}.$$
	\medskip
	
	\noindent Now, by Lemma \ref{lem:q}, $q>0$.  Thus we obtain the following upper bound.

	\begin{align*}
		\sum_{k=3}^n (qk^2-(6q-q')k) \frac{n!}{(n-k)!(2k)}&\leq \sum_{k=3}^n (qk+q')
			\frac{n!}{(n-k)!(2)}\\
			&= \frac{n!}{2}\sum_{k=3}^n \frac{qk + q'}{(n-k)!}\\
			&\leq(qn+q'n)\frac{n!}{2}\sum_{j=1}^\infty \frac{1}{j!}\\
			&\leq \left(\frac{q+q'}{2}\right)(n)n!e.
	\end{align*}

\medskip
	
	\noindent Putting these inequalities together we see, that the mean sum of squared writhe of a random linear
	embedding of $K_n$ is of order of 
	$\theta(n(n!))$.

\end{proof}

We remark that the calculations above can be modified to show that the total number of links
(resp. cycles) in a random linear embedding of $K_n$ is of the order of $\theta(\frac{n!}{n})$,
so that the mean average linking number (resp. writhe), where we average over all two component
links (resp. cycles) in $K_n$, is of the order of $\theta(n^2)$. This agrees with the results of
\cite{arsuaga07} and \cite{panagiotou10}, and shows that links (resp. cycles) of length $\theta(n)$
dominate the mean sum of linking number (resp. writhe) of the embedded graph.

\bigskip

\section{Random linear embeddings of $(n,p)$-graphs}
\medskip

\begin{thm}\label{thm:randomknot}
	For any knot $J$, the probability that a random linear embedding
	of an $(n,p)$-graph contains a cycle isotopic to $J$ goes to 1 as $n \rightarrow
	\infty$.
\end{thm}

\begin{proof}
	By Negami \cite{negami91}, there is an integer $R(J)$ such that every linear embedding of the complete graph $K_{R(J)}$ in $\mathbb{R}^3$ contains $J$.  Thus, given a set of vertices $\{v_1,v_2,\dots,v_{R(J)}\}$ 
	 in general position in the cube $C^3$, the linear embedding of $K_{R(J)}$ defined by these vertices
	necessarily contains the knot $J$ as a cycle.  Furthermore, since no cycle in $K_{R(J)}$ has length more than $R(J)$, a random linear embedding of an $(R(J),p)$-graph has probability at least
	$p^{R(J)}$ of containing the knot $J$.
	
	For any $n\geq R(J)$, by partitioning the vertices into sets of $R(J)$ vertices, we see that the probability
	that a random linear embedding of an $(n,p)$-graph
	contains $J$ is at least $1-(1-p^{R(J)})^{\lfloor n/R(J) \rfloor}$. This value goes to $1$
	as $n \rightarrow \infty$.
\end{proof}

Similar results can be obtained for other intrinsic properties of spatial embeddings
of graphs. For example,

\begin{thm}
	For an $(n,p)$-graph $G$,
	\begin{enumerate}
		\item The probability that a random linear embedding of $G$
			contains a non-trivial link of two components is at least
			$1-(1-p^6)^{\lfloor n/6 \rfloor}$. In particular, it goes to $1$ as $n \rightarrow \infty$.
		\item The probability that a random linear embedding of $G$
			contains a non-split link of three components is at least
			$1-(1-p^9)^{\lfloor n/9 \rfloor}$. In particular, it goes to $1$ as $n \rightarrow \infty$.
	\end{enumerate}
\end{thm}

\begin{proof}
	The proof is similar to that of Theorem \ref{thm:randomknot}.  For part (1), we apply the 
	result of Conway and Gordon \cite{conway83} that every embedding of $K_6$ contains a non-trivial link; and for part (2), we apply the result of Naimi and Pavelescu \cite{naimi14} that every linear embedding of $K_9$ contains a non-split link of three components.  
\end{proof}

Suppose that $n\geq 6$, $p\in (0,1)$, and $G$ is a random linear embedding of
	an $(n,p)$-graph.  Then for any $k$, $l\geq 3$ such that $k+l\leq n$, the probability that $G$ contains a pair of disjoint cycles where one is a $k$-cycle and the other is an $l$-cycle is $p^{k+l}$.  We can now modify the proof of Theorem~\ref{thm:ev_ss_link} to obtain the following.
	
\begin{thm}\label{thm:nplink}
	For any $n\geq 6$ and $p\in (0,1)$, the mean sum of squared linking numbers of a random linear embedding of an $(n,p)$-graph is of the order of $\theta(p^nn(n!))$.
\end{thm}

\begin{proof}
	We note that mean sum of squared linking numbers for a random linear embedding of an $(n,p)$-graph is
	\begin{equation*}
		\frac{q}{16} \sum_{i=6}^n p^i\frac{n!}{(n-i)!}(i-5).
	\end{equation*}
	
	Taking the term $i=n$ gives the lower bound:
	\begin{equation*}
		\frac{q}{16} p^n (n!) (n-5)\geq \frac{q}{32} p^n n(n!).
	\end{equation*}
	
	For the upper bound, we re-index over $k=n-i$, so that the sum becomes
	\begin{align*}
		n! \sum_{i=6}^n \frac{p^i (i-5)}{(n-i)!} &= n!\sum_{k=0}^{n-6} \frac{p^{n-k}(n-k-5)}{k!}\\
			&\leq n(n!)p^n \sum_{k=0}^{n-6} \frac{p^{-k}}{k!}\\
			&\leq n(n!)p^n \sum_{k=0}^\infty \frac{p^{-k}}{k!}\\
			&= n(n!)p^n e^{1/p}.
	\end{align*}
\end{proof}

\bigskip

\section{Random linear embeddings of $K_6$ and $K_{3,3,1}$}\label{sec:k6_k331}
\medskip

In this section, we apply the formulas in the proof of Theorem \ref{thm:ev_ss_link} to the
graphs $K_6$ and $K_{3,3,1}$ in order to find bounds on the probability of specific types of linking
occurring.

It follows from the formula $$\frac{q}{16} \sum_{k=3}^{n-3} \sum_{l=3}^{n-k} \frac{n!}{(n-k-l)!}$$ for the mean sum of squared linking numbers for $K_n$ in the proof of Theorem \ref{thm:ev_ss_link}, that the mean sum of squared linking numbers for $K_6$ is given by

\begin{equation*}
	\frac{q}{16} \frac{6!}{0!} = 45q.
\end{equation*}
\medskip

It has been shown independently by Hughes \cite{Hughes}, Huh and Jeon \cite{huh07}, and Nikkuni \cite{Nikkuni} that every linear embedding of $K_6$ contains
either exactly one or three Hopf links, and all other links in the embedding are trivial. Hence, for a given linear embedding of $K_6$, the sum of
squared linking numbers is either 1 or 3.  

This means that $45q = p_1 + 3p_3$, where $p_1$ is the probability that a
random linear embedding of $K_6$ has exactly one Hopf link, and $p_3 = 1-p_1$ is
the probability that a random linear embedding of $K_6$ has exactly three Hopf links.
This implies that
\begin{equation*}
	p_1 = \frac{3-45q}{2}.
\end{equation*}
\medskip

We estimated the value of $q$ to be $0.033867 \pm 0.000013$ (see the numerical
computation described in Appendix \ref{sec:q}).  This value is consistent with
the value of $q = 0.0338\pm 0.024$ obtained by \cite{arsuaga07}, which is verified
in \cite{panagiotou10}.  Thus we find that $p_1 = 0.7380 \pm 0.0003$.

For the complete tripartite graph $K_{3,3,1}$, Naimi and Pavelescu \cite{naimi15} show that every linear embedding
contains either $1,2,3,4$, or $5$ non-trivial links. Furthermore, they show that if the number of non-trivial links is
odd, all such links are Hopf links; whereas if the
number of non-trivial links is even, then one link is a $(2,4)$-torus link and the rest are Hopf links.  Since a Hopf link has linking number $\pm 1$ and a $(2,4)$-torus link has linking number $\pm 2$, it follows that the sum of squared linking numbers for any linear embedding of $K_{3,3,1}$ is either $1$, $3$, $5$ or $7$.

Now every pair of disjoint cycles in $K_{3,3,1}$ consists of one $3$-cycle and one $4$-cycle.  By Lemma~\ref{lem:ev_mn_link}, the mean squared linking number of a random linear embedding of a disjoint union of a $3$-cycle and a $4$-cycle is $\frac{q}{2}(3)(4)$.  Since there
are nine pairs of disjoint cycles in $K_{3,3,1}$, it follows that the expected value of the sum of squared linking numbers of a linear embedding of $K_{3,3,1}$ is

\begin{equation*}
	\frac{q}{2} (9)(3)(4) = 54 q.
\end{equation*}

\medskip

For each $k$, we let $p_k$ be the probability that there are $k$ non-trivial links in the embedding. Then, this expected value is equal to

\begin{equation*}
	1p_1 + 5p_2 + 3p_3 + 7p_4 + 5p_5 \geq p_1 + 3(1-p_1).
\end{equation*}
\medskip

Hence, it follows that the probability that there is precisely one non-trivial link in a random linear embedding of $K_{3,3,1}$ is given by 

\begin{equation*}
	p_1 \geq \frac{3-54q}{2} = 0.5856 \pm 0.0004.
\end{equation*}

\bigskip

\section{Experimental data}

In this section, we describe some experimental results we obtained for links in random linear embeddings of
graphs.

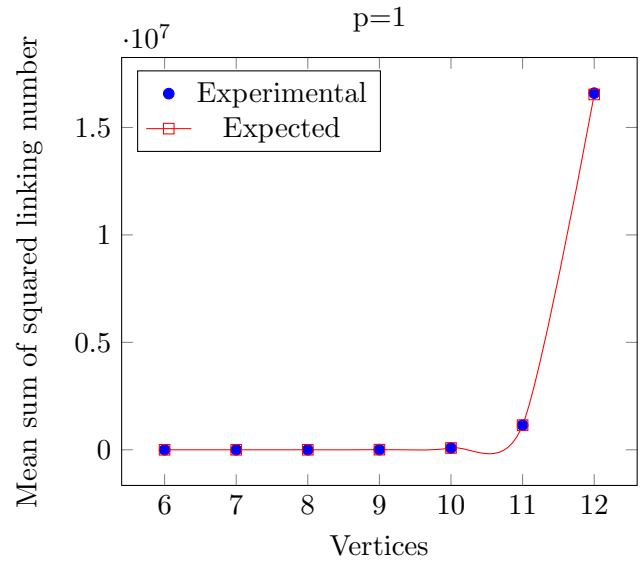
\begin{figure}
\begin{tikzpicture}
	\begin{axis}[
		xlabel=Vertices,
		ylabel=Mean sum of squared linking number,
		legend pos = north west,
		title = p{=}1]
	\addplot[only marks,mark=*,blue]
		plot coordinates {
			(6,1.542)
			(7,32.13)
			(8,470.844)
			(9,6269.844)
			(10,83741.85)
			(11,1146853.96)
			(12,16594972.7)
		};
	\addlegendentry{Experimental}
	
	\addplot[smooth,mark=square,red]
		plot coordinates {
			(6,1.52402)
			(7,32.0043)
			(8,469.397)
			(9,6272.85)
			(10,83531.3)
			(11,1148380)
			(12,16536400)
		};
	\addlegendentry{Expected}

	\end{axis}
\end{tikzpicture}
\caption{Experimental vs. expected mean sum of squared linking number for $p=1$. Experimental
data used 1000 samples for $6\leq n \leq 11$, and 100 samples for $n=12$.}\label{fig:p=1}
\end{figure}

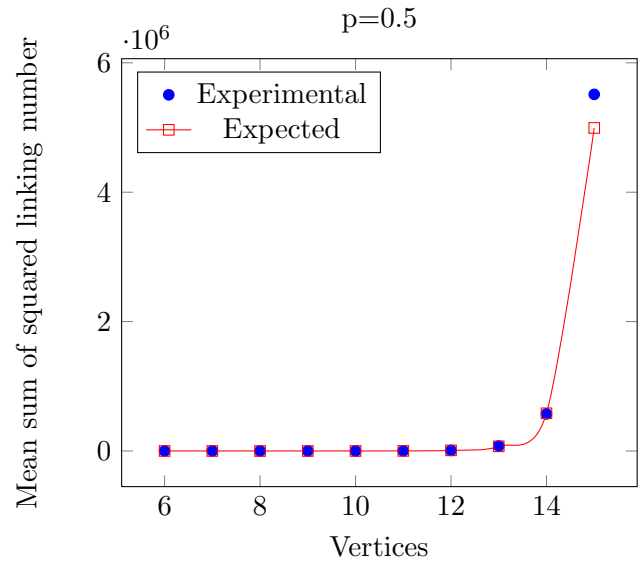
\begin{figure}
\begin{tikzpicture}
	\begin{axis}[
		xlabel=Vertices,
		ylabel=Mean sum of squared linking number,
		legend pos = north west,
		title = p{=}0.5]	
	\addplot[only marks,color=blue,mark=*]
		plot coordinates {
			(6,0.023)
			(7,0.345)
			(8,2.633)
			(9,20.531)
			(10,170.59)
			(11,1284.983)
			(12,8806.644)
			(13,76533.147)
			(14,573495.272)
			(15,5512620.35)
		};
	\addlegendentry{Experimental}
	
	\addplot[smooth,color=red,mark=square]
		plot coordinates {
			(6,0.0238127)
			(7,0.333378)
			(8,3.0004)
			(9,23.0031)
			(10,167.523)
			(11,1221.16)
			(12,9147.73)
			(13,71336)
			(14,582553)
			(15,4993280)
		};
	\addlegendentry{Expected}

	\end{axis}
\end{tikzpicture}
\caption{Experimental vs. expected mean sum of squared linking number for $p=0.5$. Experimental
data used 1000 samples for $6\leq n \leq 14$, and 100 samples for $n=15$.}\label{fig:p=0.5}
\end{figure}

\begin{figure}
\begin{tikzpicture}
	\begin{axis}[
		xlabel=Vertices,
		ylabel=Mean sum of squared linking number,
		legend pos = north west,
		title = p{=}0.25]
	\addplot[only marks,mark=*,blue]
		plot coordinates {
			(6,0.000445)
			(7,0.00399)
			(8,0.0249)
			(9,0.12584)
			(10,0.57158)
			(11,2.40882)
			(12,10.10214)
			(13,40.8)
			(14,170.7914)
			(15,861.065)
			(16,3859.7)
			(17,20933.63)
		};
	\addlegendentry{Experimental}
	
	\addplot[smooth,mark=square,red]
		plot coordinates {
			(6,0.000372074)
			(7,0.00390678)
			(8,0.0247429)
			(9,0.125017)
			(10,0.564041)
			(11,2.41463)
			(12,10.1783)
			(13,43.2545)
			(14,188.121)
			(15,845.054)
			(16,3941.6)
			(17,19142.3)
		};
	\addlegendentry{Expected}

	\end{axis}
\end{tikzpicture}
\caption{Experimental vs. expected mean sum of squared linking number for $p=0.25$. Experimental
data used 200000 samples for $6 \leq n \leq 7$, 50000 samples for $8\leq n \leq 12$, 5000 samples for
$13 \leq n \leq 15$, 500 samples for $n=16$, and 100 samples for $n=17$ .}\label{fig:p=0.25}
\end{figure}
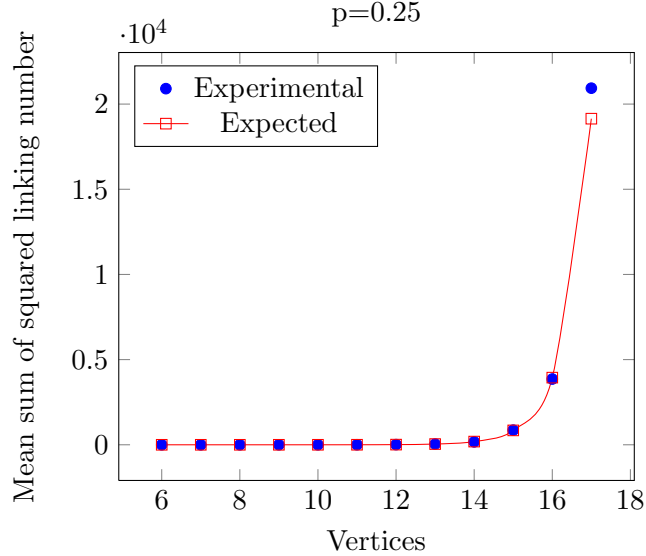

The data was generated using a Python program, taking
coordinates of the $n$ vertices to be uniformly distributed in (0,1). An edge between two vertices
is taken with probability $p$, and then the number of links with each linking number are
tallied. To give a more accurate picture of the distribution of linking numbers and the
average sum of linking numbers, we took multiple samples for each $(n,p)$.

We first investigated the mean sum of squared linking number for $p=1$, $p=0.5$, and $p=0.25$,
comparing experimental data with the expected number from the formula given in the proof of Theorem \ref{thm:nplink}.
For $p=1$, we took 1000 samples for $6 \leq n \leq 11$, and $100$ samples for $n=12$ (see
Figure \ref{fig:p=1}) For $p=0.5$, we took 1000 samples for $6 \leq n \leq 14$ and 100 samples for
$n=15$ (see Figure \ref{fig:p=0.5}). For $p=0.25$, we took 200,000 samples for $6 \leq n \leq 7$,
50,000 samples for $8 \leq n \leq 12$,
5000 samples for $13 \leq n \leq 15$, 500 samples for $n=16$, and 100 samples for $n=17$ (see
Figure \ref{fig:p=0.25}).

The experimental data follows the expected super-factorial growth. The deviation for the
mean sum of squared linking number is within approximately 10\%
of the expected value (and most data points are within 5\%), and the discrepancy for large
numbers of vertices and $p<1$ is due to the small number of samples taken due to
computational constraints. In addition, from the $n=6$, $p=1$ case, we can determine that of the 1000
random linear embeddings sampled, 729 had exactly one Hopf link, giving a 99\% confidence interval
for the probability that a random linear embedding of $K_6$ has one Hopf link of $0.729 \pm
0.036$, which agrees with the theoretical computation from Section \ref{sec:k6_k331}
and numerical value of $q$.

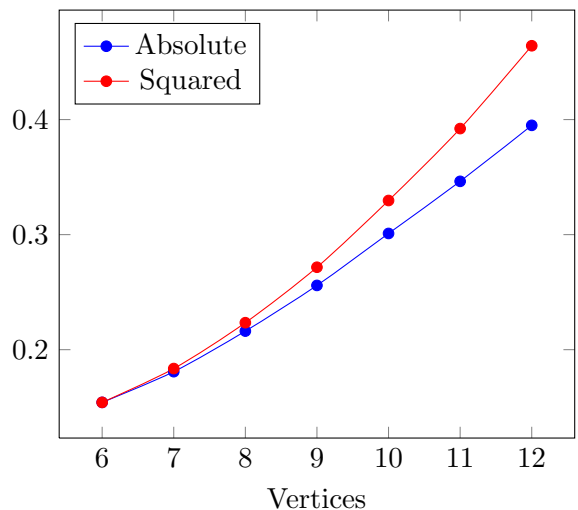
\begin{figure}
	\begin{tikzpicture}
		\begin{axis}[
			xlabel=Vertices,
			title=Mean average squared and absolute linking number,
			legend pos = north west]
		\addplot[smooth,mark=*,blue]
			plot coordinates {
				(6,0.1542)
				(7,0.18089143)
				(8,0.21615187)
				(9,0.25587044)
				(10,0.30102775)
				(11,0.34632512)
				(12,0.39499741)
			};
		\addlegendentry{Absolute}
		
		\addplot[smooth,mark=*,red]
			plot coordinates {
				(6,0.1542)
				(7,0.1836)
				(8,0.22346654)
				(9,0.27166879)
				(10,0.32966766)
				(11,0.3921902)
				(12,0.46417805)
			};
		\addlegendentry{Squared}
		\end{axis}
	\end{tikzpicture}
	
	\caption{Mean average squared and absolute linking number for $p=1$.}
	\label{fig:abs_p=1}
\end{figure}

\begin{figure}
	\begin{tikzpicture}
		\begin{axis}[
			xlabel=Vertices,
			title=Mean average squared and absolute linking number,
			legend pos = north west]
		\addplot[smooth,mark=*,blue]
			plot coordinates {
				(6,0.15862069)
				(7,0.17109885)
				(8,0.18931577)
				(9,0.22117335)
				(10,0.26123189)
				(11,0.30642165)
				(12,0.34530022)
				(13,0.39758039)
				(14,0.452373)
				(15,0.51551164)
			};
		\addlegendentry{Absolute}
		
		\addplot[smooth,mark=*,red]
			plot coordinates {
				(6,0.15862069)
				(7,0.17310587)
				(8,0.19268203)
				(9,0.23266435)
				(10,0.28072941)
				(11,0.33923057)
				(12,0.39304373)
				(13,0.47054833)
				(14,0.55837714)
				(15,0.66909645)
			};
		\addlegendentry{Squared}
		\end{axis}
	\end{tikzpicture}
	
	\caption{Mean average squared and absolute linking number for $p=0.5$.}
	\label{fig:abs_p=0.5}
\end{figure}
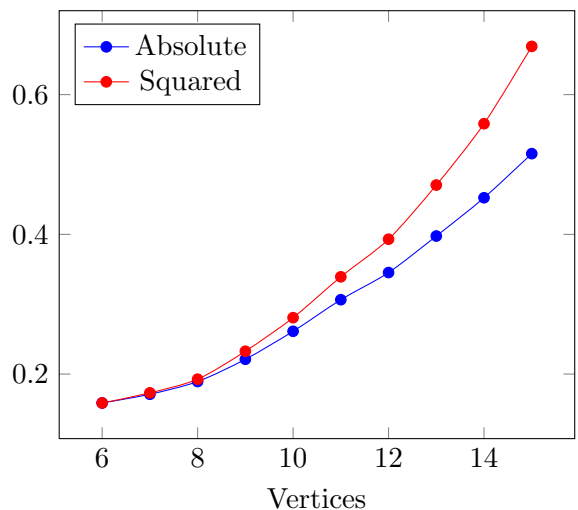

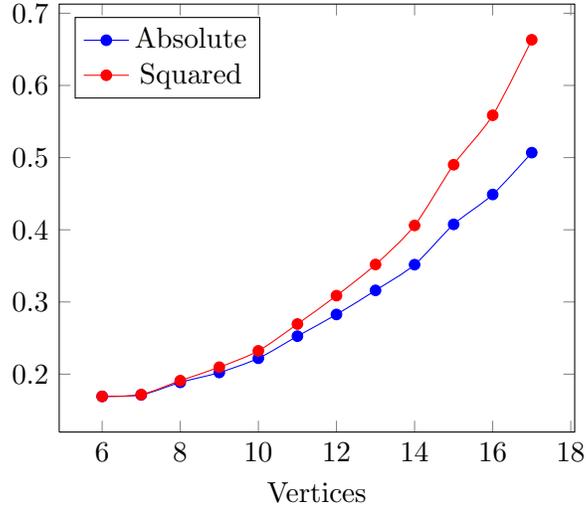
\begin{figure}
	\begin{tikzpicture}
		\begin{axis}[
			xlabel=Vertices,
			title=Mean average squared and absolute linking number,
			legend pos = north west]
		\addplot[smooth,mark=*,blue]
			plot coordinates {
				(6,0.16920152)
				(7,0.17129331)
				(8,0.18858327)
				(9,0.20228537)
				(10,0.22199017)
				(11,0.25260659)
				(12,0.28269296)
				(13,0.31606748)
				(14,0.351712589)
				(15,0.40751253)
				(16,0.44890358)
				(17,0.50687237)
			};
		\addlegendentry{Absolute}
		
		\addplot[smooth,mark=*,red]
			plot coordinates {
				(6,0.16920152)
				(7,0.17172369)
				(8,0.19103882)
				(9,0.20961455)
				(10,0.23224574)
				(11,0.26947671)
				(12,0.30877361)
				(13,0.35190677)
				(14,0.40603675)
				(15,0.49012264)
				(16,0.55865154)
				(17,0.66325654)
			};
		\addlegendentry{Squared}
		\end{axis}
	\end{tikzpicture}
	
	\caption{Mean average squared and absolute linking number for $p=0.25$.}
	\label{fig:abs_p=0.25}
\end{figure}

In addition, we computed the mean average squared linking number and the mean average
absolute linking number, where the average is taken over all links in the graph, and then
the mean is taken over all samples of a given size. The experimental data for the mean average
squared linking number follows a quadratic growth, as expected. From the samples we computed,
it appeared that the average absolute linking number was also quadratic, which differs
from the linear growth rate for absolute linking number of random polygons studied in
\cite{arsuaga07} and \cite{panagiotou10} (see Figures \ref{fig:abs_p=1}, \ref{fig:abs_p=0.5},
and \ref{fig:abs_p=0.25}). However, because of the small number of data points
that we could compute, this is inconclusive.

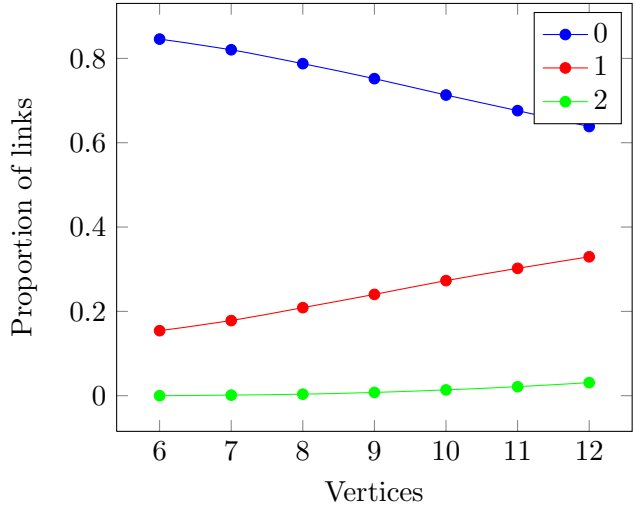
\begin{figure}
\begin{tikzpicture}
	\begin{axis}[
		xlabel=Vertices,
		ylabel=Proportion of links]
	\addplot[smooth,mark=*,blue]
		plot coordinates {
			(6,0.8458)
			(7,0.820462857)
			(8,0.787497864)
			(9,0.751976299)
			(10,0.713103823)
			(11,0.676078728)
			(12,0.638369158)
		};
	\addlegendentry{0}
	
	\addplot[smooth,mark=*,red]
		plot coordinates {
			(6,0.1542)
			(7,0.178182857)
			(8,0.208859991)
			(9,0.240229083)
			(10,0.272951386)
			(11,0.302038277)
			(12,0.329460525)
		};
	\addlegendentry{1}
	
	\addplot[smooth,mark=*,green]
		plot coordinates {
			(6,0)
			(7, 0.0013542857)
			(8,0.0036345515)
			(9,0.0077427965)
			(10,0.0137596085)
			(11,0.0213699091)
			(12,0.031001266)
		};
	\addlegendentry{2}
	
	\end{axis}
\end{tikzpicture}
\caption{Proportion of links in random linear embeddings of $K_n$ with linking number 0, 1, and 2.}
\label{fig:prop_p=1}
\end{figure}

\begin{figure}
\begin{tikzpicture}
	\begin{axis}[
		xlabel=Vertices,
		ylabel=Proportion of links]
	\addplot[smooth,mark=*,blue]
		plot coordinates {
			(6,0.84137931)
			(7,0.829904666)
			(8,0.812367362)
			(9,0.78453815)
			(10,0.748432941)
			(11,0.709669704)
			(12,0.677923769)
			(13,0.637439861)
			(14,0.597673808)
			(15,0.555268433)
		};
	\addlegendentry{0}
	
	\addplot[smooth,mark=*,red]
		plot coordinates {
			(6,0.15862069)
			(7,0.169091821)
			(8,0.185949506)
			(9,0.209784346)
			(10,0.241982862)
			(11,0.27454755)
			(12,0.299486929)
			(13,0.328969061)
			(14,0.355121606)
			(15,0.379621717)
		};
	\addlegendentry{1}
	
	\addplot[smooth,mark=*,green]
		plot coordinates {
			(6,0)
			(7, 0.0010035123)
			(8,0.0016831321)
			(9,0.0056435071)
			(10,0.0095068516)
			(11,0.0154786224)
			(12,0.021967601)
			(13,0.0321962452)
			(14,0.044472157)
			(15,0.0597652854)
		};
	\addlegendentry{2}
	
	\end{axis}
\end{tikzpicture}
\caption{Proportion of links in random linear embeddings of $(n,p)$ graphs with linking number 0, 1, and 2, when $p=0.5$.}
\label{fig:prop_p=0.5}
\end{figure}
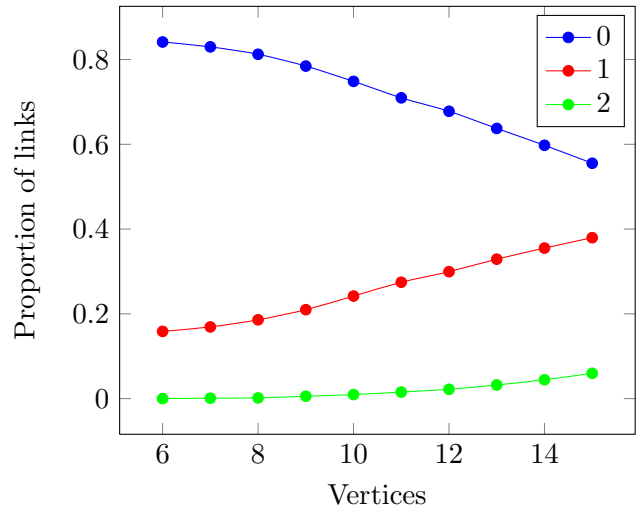

\begin{figure}
\begin{tikzpicture}
	\begin{axis}[
		xlabel=Vertices,
		ylabel=Proportion of links]
	\addplot[smooth,mark=*,blue]
		plot coordinates {
			(6,0.830798479)
			(7,0.828921885)
			(8,0.812643855)
			(9,0.801379218)
			(10,0.783113242)
			(11,0.755705413)
			(12,0.730082667)
			(13,0.701429535)
			(14,0.674444358)
			(15,0.631532668)
			(16,0.6024813)
			(17,0.564692735)
		};
	\addlegendentry{0}
	
	\addplot[smooth,mark=*,red]
		plot coordinates {
			(6,0.169201521)
			(7,0.170862922)
			(8,0.186128587)
			(9,0.194956191)
			(10,0.211807728)
			(11,0.236103404)
			(12,0.257398452)
			(13,0.281484012)
			(14,0.300332073)
			(15,0.331579878)
			(16,0.349439669)
			(17,0.369996854)
		};
	\addlegendentry{1}
	
	\addplot[smooth,mark=*,green]
		plot coordinates {
			(6,0)
			(7, 0.002151926)
			(8,0.0012275587)
			(9,0.0036645901)
			(10,0.0050546504)
			(11,0.0080725997)
			(12,0.0122700812)
			(13,0.0166879708)
			(14,0.024314933)
			(15,0.0348267038)
			(16,0.0449491805)
			(17,0.0594121581)
		};
	\addlegendentry{2}
	
	\end{axis}
\end{tikzpicture}
\caption{Proportion of links in random linear embeddings of $(n,p)$ graphs with linking number 0, 1, and 2, when $p=0.25$.}
\label{fig:prop_p=0.25}
\end{figure}
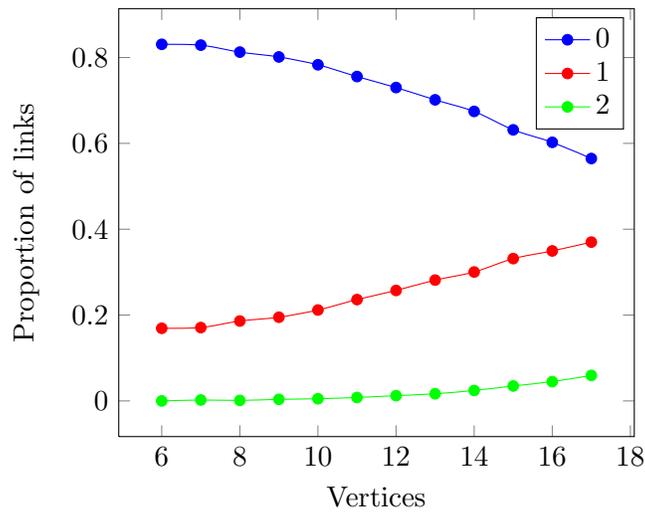

We also investigated the distributions of links with a given linking number in
random $(n,p)$ graphs (see Figures \ref{fig:prop_p=1}, \ref{fig:prop_p=0.5}, and
\ref{fig:prop_p=0.25}). Omitted from the figures are links with linking number greater
than $2$, which were detected in fewer than $1\%$ of the links in the samples that
we generated. We expect that for a fixed linking number $k$, the proportion of
links with that linking number will increase, peak, then decrease as $n \rightarrow \infty$.
However, due to computational constraints, we were not able to compute samples out
to large enough $n$ to see this behavior, even with linking number 1.

\bigskip

\appendix

\section{Computing $q$} \label{sec:q}

In order to obtain better estimates of the linking probabilities for $K_6$ and
$K_{3,3,1}$, we numerically estimated the value of $q$.  Consider two triangles described by the
consistently oriented edges $l_1,l_2,l_3$ and $l_1',l_2',l_3'$, and let $\epsilon_{ij}$
denote the signed crossing number between $l_i$ and $l_j'$.  Then it follows from the proof of Lemma \ref{lem:q} in \cite{arsuaga07} that 

\begin{equation*}
	E\left[\left(\sum_{i,j=1}^3 \epsilon_{ij}\right)^2\right] = 18q.
\end{equation*}
\medskip

But the quantity $\sum_{i,j=1}^3 \epsilon_{ij}$ is precisely twice the linking number
of the two triangles, which for a linear embedding is either $0$ or $\pm 1$. Hence,
$\frac{18q}{4}$ is the probability that a random linear embedding of two disjoint
triangles is linked. 

We wrote a Python program to generate random linear
embeddings of two triangles by taking six random points in $C^3$, described as
three (pseudo)random coordinates in $(0,1)$, and computing
the associated linking number. Out of the one billion samples generated,
152,402,780 were linked, giving a
99\% confidence interval for $q$ of $0.033867 \pm 0.000013$. 

\medskip

All code is available at  \url{http://math.berkeley.edu/~kozai/random_graphs/}.
\bigskip

\section*{Acknowledgements}

The first author would like to thank the Mathematical Science Research Institute for its hospitality during the Spring of 2015 which facilitated this collaboration.

\bigskip

\bibliographystyle{amsplain}
\bibliography{randomgraphs-nomr.bib}

\providecommand{\bysame}{\leavevmode\hbox to3em{\hrulefill}\thinspace}
\providecommand{\MR}{\relax\ifhmode\unskip\space\fi MR }
\providecommand{\MRhref}[2]{%
  \href{http://www.ams.org/mathscinet-getitem?mr=#1}{#2}
}
\providecommand{\href}[2]{#2}
\begin{thebibliography}{10}

\bibitem{arsuaga07}
J.~Arsuaga, T.~Blackstone, Y.~Diao, E.~Karadayi, and M.~Saito, \emph{Linking of
  uniform random polygons in confined spaces}, J. Phys. A \textbf{40} (2007),
  no.~9, 1925--1936.

\bibitem{Arsuaga09}
J.~Arsuaga, B.~Borgo, Y.~Diao, and R.~Scharein, \emph{The growth of the mean
  average crossing number of equilateral polygons in confinement}, J. Phys. A
  \textbf{42} (2009), no.~46, 465202, 9.

\bibitem{conway83}
J.~H. Conway and C.~McA. Gordon, \emph{Knots and links in spatial graphs}, J.
  Graph Theory \textbf{7} (1983), 445--453.

\bibitem{Diao}
Y.~Diao, C.~Ernst, S.~Saarinen, and U.~Ziegler, \emph{Generating random walks
  and polygons with stiffness in confinement}, J. Phys. A \textbf{48} (2015),
  no.~9, 095202, 19.

\bibitem{Diao93}
Y.~Diao, N.~Pippenger, and D.W. Sumners, \emph{On random knots}, Random
  knotting and linking ({V}ancouver, {BC}, 1993), Ser. Knots Everything,
  vol.~7, World Sci. Publ., River Edge, NJ, 1994, pp.~187--197.

\bibitem{gilbert59}
E.~N. Gilbert, \emph{Random graphs}, Ann. Math. Statist. \textbf{30} (1959),
  1141--1144.

\bibitem{Hughes}
C.~Hughes, \emph{Linked triangle pairs in straight edge embeddings of {$K_6$}},
  Pi Mu Epsilon Journal \textbf{12} (2006), no.~4, 213--218.

\bibitem{huh07}
Y.~Huh and C.~B. Jeon, \emph{Knots and links in linear embeddings of {$K_6$}},
  J. Korean Math. Soc. \textbf{44} (2007), no.~3, 661--671.

\bibitem{naimi14}
R.~Naimi and E.~Pavelescu, \emph{Linear embeddings of {$K_9$} are triple
  linked}, J. Knot Theory Ramifications \textbf{23} (2014), no.~3, 1420001, 9.

\bibitem{naimi15}
\bysame, \emph{On the number of links in a linearly embedded {$K_{3,3,1}$}}, J.
  Knot Theory Ramifications (2015), to appear.

\bibitem{negami91}
S.~Negami, \emph{Ramsey theorems for knots, links and spatial graphs}, Trans.
  Amer. Math. Soc. \textbf{324} (1991), no.~2, 527--541.

\bibitem{Nikkuni}
R.~Nikkuni, \emph{A refinement of the {C}onway-{G}ordon theorems}, Topology
  Appl. \textbf{156} (2009), no.~17, 2782--2794.

\bibitem{panagiotou10}
E.~Panagiotou, K.~C. Millett, and S.~Lambropoulou, \emph{The linking number and
  the writhe of uniform random walks and polygons in confined spaces}, J. Phys.
  A \textbf{43} (2010), no.~4, 045208, 28.

\bibitem{Portillo}
J.~Portillo, Y.~Diao, R.~Scharein, J.~Arsuaga, and M.~Vazquez, \emph{On the
  mean and variance of the writhe of random polygons}, J. Phys. A \textbf{44}
  (2011), no.~27, 275004, 19.

\bibitem{Numerical}
K.~Tsurusaki and T.~Deguchi, \emph{Numerical analysis on topological
  entanglements of random polygons}, Statistical models, {Y}ang-{B}axter
  equation and related topics, and {S}ymmetry, statistical mechanical models
  and applications ({T}ianjin, 1995), World Sci. Publ., River Edge, NJ, 1996,
  pp.~320--329.

\end{thebibliography}

\end{document}